\documentclass[psamsfont]{amsart}
\usepackage{graphicx}
\usepackage[dvips]{epsfig}
\usepackage{pinlabel}
\usepackage{amsmath}
\usepackage{amsfonts}
\usepackage{latexsym}
\usepackage{amssymb}
\usepackage[usenames]{color}
\usepackage{amsthm}

\input xy
\xyoption{all}

\author{Baptiste Chantraine }
\address{Universit\'e de Nantes, France.}
\email{baptiste.chantraine@univ-nantes.fr}
\theoremstyle{remark}
\newtheorem{defn}{Definition}[section]
\newtheorem{Rem}[defn]{Remark}

\theoremstyle{plain}
\newtheorem{theorem}{Theorem}[section]

\DeclareMathAlphabet{\mathdj}{U}{msb}{m}{n}

\begin{document}

\title{Lagrangian concordance is not a symmetric relation.}
\thispagestyle{empty}

\maketitle

\begin{abstract}
  We provide an explicit example of a non trivial Legendrian knot $\Lambda$  such that there exists a Lagrangian concordance from $\Lambda_0$ to $\Lambda$ where $\Lambda_0$ is the trivial Legendrian knot with maximal Thurston-Bennequin number. We then use the map induced in Legendrian contact homology by a concordance  and the augmentation category of $\Lambda$ to show that no Lagrangian concordance exists in the other direction. This proves that the relation of Lagrangian concordance is not symmetric.
\end{abstract}

\section{Introduction.}
\label{sec:introduction}
In this paper we will only consider the standard contact $\mathbb{R}^3$ with the contact structure $\xi=\ker\alpha$ with $\alpha=dz-ydx$. A Legendrian knot is an embedding $i:S^1\hookrightarrow\mathbb{R}^3$ such that $i^*\alpha=0$. The symplectisation of $(\mathbb{R}^3,\xi)$ is the symplectic manifold $(\mathbb{R}\times\mathbb{R}^3,d(e^t\alpha))$. 

In \cite{chantraine_conc} we introduced the notion of Lagrangian concordances and cobordisms between Legendrian knots and proved the basic properties of those relations. Roughly speaking a Lagrangian cobordism $\Sigma$ from a knot $\Lambda^-$ to a knot $\Lambda^+$ is a Lagrangian submanifold of the symplectisation which coincides at $-\infty$ with $\Lambda^-$ and at $+\infty$ with $\Lambda^+$. When $\Sigma$ is topologically a cylinder we say that $\Lambda^-$ is Lagrangian concordant to $\Lambda^+$ (a relation we denote by $\Lambda^-\prec \Lambda^+$). Among the basic properties of oriented Lagrangian cobordisms we proved that $tb(\Lambda^+)-tb(\Lambda^-)=2g(\Sigma)$ where $tb(\Lambda)$ is the Thurston-Bennequin number of  $\Lambda$. This immediately implies that when a Lagrangian cobordism is not a cylinder then such a cobordism cannot be reversed. However we cannot apply such an argument to explicitly prove that the relation of concordance is not symmetric. In this paper we use more involved techniques, in particular recent results of T. Ekholm, K. Honda and T. K\'alm\'an in \cite{Ekhoka} using pseudo-holomorphic curves and Legendrian contact homology, to give an example of a non reversible Lagrangian concordance. Namely we prove that:

\begin{theorem}\label{thm:principal}
Let $\Lambda_0$ be the Legendrian unknot with $-1$ Thurston-Bennequin invariant.
  There exists a Legendrian representative $\Lambda$ of the knot $m(9_{46})$ of Rolfsen table of knots (see \cite{Rolfsen_Knots}) such that: 
  \begin{itemize}
  \item $\Lambda_0\prec\Lambda$
\item $\Lambda\not\prec\Lambda_0$.
  \end{itemize}
\end{theorem}

The front and Lagrangian projections of $\Lambda$ in the previous theorem are shown on Figure \ref{fig:frontlag} (note that this Legendrian knot also appears in the end of \cite{Sivek}  as an example of Lagrangian slice knot). 
\begin{figure}[!h]
  \centering
  \includegraphics[height=7cm]{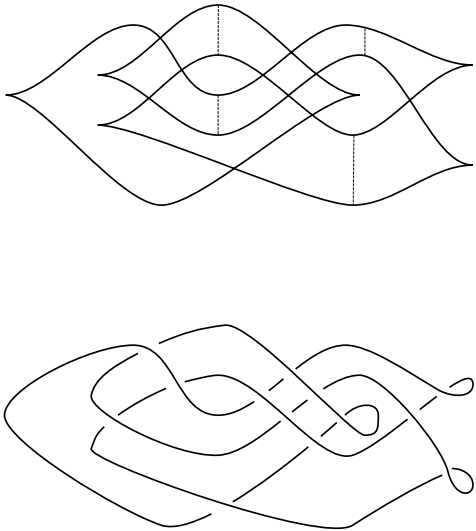}
  \caption{Front and Lagrangian projections of a Legendrian representative of $m(9_{46})$.}
  \label{fig:frontlag}
\end{figure}

This example confirms the analogy of this relation with a partial order. Whether or not it is a genuine partial order (meaning that $\Lambda\prec\Lambda'$ and $\Lambda'\prec \Lambda$ would imply that $\Lambda$ is Legendrian isotopic to $\Lambda'$) is neither proved nor disproved; the author is unaware of any conjecture on how different the equivalence relation given by $\Lambda\prec\Lambda'$ and $\Lambda'\prec \Lambda$ is from the Legendrian isotopy relation.

The knot $\Lambda$ is the ``smallest'' Lagrangianly slice Legendrian knot (as it is clear from the Legendrian knot atlas of \cite{atlas}); it is therefore the first natural candidate fo an example a non-reversible concordance. Using connected sums it is possible to construct more examples of this kind. Another class of examples in dimension $3$ will appear in forthcoming work by J. Baldwin and S. Sivek in \cite{SivekBaldwin} where they construct concordances where the negative ends are stabilisations and the positive ones have non-vanishing Legendrian contact homology. In higher dimensions recent results of Y. Eliashberg and E. Murphy \cite{Eliashmurphy} imply that if the negative end is loose (in the sense of \cite{Murphy}) then the Lagrangian concordance problem satisfies the h-principle. This can be used to prove further non reversible examples of Lagrangian concordances. Note that in both of those cases we still need pseudo-holomorphic curves techniques and the existence of maps in Legendrian contact homology to prove that the involved Lagrangian concordances cannot be reversed.

In order to prove the existence of the Lagrangian concordance claimed in Theorem \ref{thm:principal} we use elementary Lagrangian cobordisms from \cite{collarable} which we recall in Section \ref{sec:elem-lagr-cobord}. We also describe those elementary cobordisms in terms of Lagrangian projections as we will use those in Section \ref{sec:legendr-cont-homol} to compute maps between Legendrian contact homology algebras (LCH for short). As the negative end of the concordance is $\Lambda_0$ which has non-vanishing LCH the actual argument not only relies on the functoriality of Legendrian contact homology (as it is the case for the example of \cite{Eliashmurphy} and \cite{SivekBaldwin}) but also on a unknottedness result of Lagrangian concordances from $\Lambda_0$ to itself which follows from work of Y. Eliashberg and L. Polterovitch in \cite{Eliashberg_&_Local_Lagrangian_knots} which we state in the following:

\begin{theorem}\label{thm:lagunknot}
Consider the standard contact $S^3$ (seen as the compactification of the standard contact $\mathbb{R}^3$) and denote by $K_0$ the Legendrian unknot with $-1$ Thurston-Bennequin invariant (which corresponds to $\Lambda_0$ in $\mathbb{R}^3$).

  Let $C$ be an oriented Lagrangian cobordism from $K_0$ to itself. Then there is a compactly supported symplectomorphism of $\mathbb{R}\times S^3$ such that $\phi(C)=\mathbb{R}\times K_0$.
\end{theorem}

Theorem \ref{thm:lagunknot} is proven in Section \ref{sec:lagr-conc-from}. Assuming then that a concordance $C'$ from $\Lambda$ to $\Lambda_0$ exists we could glue $C$ to $C'$ to get a concordance from $\Lambda_0$ to $\Lambda_0$ and applying  Theorem \ref{thm:lagunknot} we deduce that the map induced in Legendrian contact homology is the identity (as stated in Theorem \ref{thm:concmaptriv}). We conclude the proof of Theorem \ref{thm:principal} in Section \ref{sec:non-symetry-th}. In order to do so, we use the augmentation categories of $\Lambda$ and $\Lambda_0$ as defined in \cite{augcat} and the functor between them induced by the concordance to find a contradiction to the existence of a concordance from $\Lambda$ to $\Lambda_0$. 

\begin{Rem}
The main result was announced in the addendum in the introduction of \cite{chantraine_conc}. When it was written bilinearised LCH was not known to the author. The original proof of the non-symmetry followed however similar lines. The idea is to construct several other concordances $C_i$ from $\Lambda_0$ to $\Lambda$ (every dashed line in Figure \ref{fig:frontlag} is a chord where we can apply move number $4$ of Figure \ref{fig:localmove} to get such a concordance). For each of those we computed the associated map similarly to what is done in Section \ref{sec:legendr-cont-homol}. We then used Theorem \ref{thm:concmaptriv} to prove that for each of them the composite map in Legendrian contact homology is the identity and deduce after some effort a contradiction. The existence of the augmentation category allows us to give a more direct final argument and use only one explicit concordance from $\Lambda_0$ to $\Lambda$. 
\end{Rem}

{\bf Acknowledgements.} Most of this work was done while the author was supported first by a post-doctoral fellowship and after by a Mandat Charg\'e de Recherche from the Fonds de la Recherche Scientifique (FRS-FNRS), Belgium. I wish to thank both the FNRS and the mathematics department of the Universit\'e Libre de Bruxelles for the wonderful work environment they provided. I also thank two anonymous referees whose comments and suggestions improved the exposition of the paper.

\section{Lagrangian concordances and Legendrian contact homology.}
\label{sec:defin-prev-results}

We recall in this section the main definition from \cite{chantraine_conc}. 

\begin{defn}\label{def:lagconc}
  Let $\Lambda^-:S^1\hookrightarrow\mathbb{R}^3$ and $\Lambda^+:S^1\hookrightarrow\mathbb{R}^3$ be two Legendrian knots in $\mathbb{R}^3$. We say that $\Lambda^-$ is \textit{Lagrangian concordant} to $\Lambda^+$ if there exists a Lagrangian embedding $C:\mathbb{R}\times\Lambda\hookrightarrow\mathbb{R}\times\mathbb{R}^3$ such that 
  \begin{enumerate}
  \item $C\vert_{(-\infty,-T)\times\Lambda}=Id\times\Lambda^-$.
\item $C\vert_{(T,\infty)\times\Lambda}=Id\times\Lambda^+$.
  \end{enumerate}
In this situation $C$ is called a \textit{Lagrangian concordance} from $\Lambda^-$ to $\Lambda^+$.
\end{defn}

It was proven in \cite{findimlagint} that two Legendrian isotopic Legendrian knots are indeed Lagrangian concordant. Another proof is given in \cite{chantraine_conc} where we also proved that under Lagrangian concordances the classical invariants $tb$ and $r$ are preserved.

A Lagrangian concordance $C$ is always an exact Lagrangian submanifold of $\mathbb{R}\times\mathbb{R}^3$ in the sense of \cite{Ekhoka} and thus following \cite{Ekhoka} it defines a DGA-map $$\varphi_C:\mathcal{A}(\Lambda^+)\rightarrow\mathcal{A}(\Lambda^-)$$ where $\mathcal{A}(\Lambda^\pm)$ denote the Chekanov algebras of the Legendrian submanifolds $\Lambda^\pm$. The homology of $\mathcal{A}(\Lambda)$ (denoted by $LCH(\Lambda)$) is called the Legendrian contact homology of $\Lambda$ (see \cite{Chekanov_DGA_Legendrian} and \cite{Ekholm_Contact_Homology}). This map is defined by a count of pseudo-holomorphic curves with boundary on $C$. 

If $C_1$ is a Lagrangian concordance from $\Lambda_0$ to $\Lambda_1$ and $C_2$ a Lagrangian concordance from $\Lambda_1$ to $\Lambda_2$. We denote by $C_1\#_T C_2$ the Lagrangian concordance from $\Lambda_0$ to $\Lambda_2$ which is equal to a translation of $C_1$ for $t<-T$ and a translation of $C_2$ for $t>T$. Then \cite[Theorem 1.2]{Ekhoka} implies that there exists a sufficiently big $T$ such that $\varphi_{C_1\#_T C_2}=\varphi_{C_1}\circ \varphi_{C_2}$, in particular the association $C\rightarrow \varphi_C$ is functorial on LCH.

\section{Elementary Lagrangian cobordisms and their Lagrangian projections.}
\label{sec:elem-lagr-cobord}
For a Legendrian knot $\Lambda$ in $\mathbb{R}^3$ we call the projection of $\Lambda$ on the $xz$-plane along the $y$ direction the \textit{front projection} of $\Lambda$. The projection on the $xy$ plane along the $z$ direction is called the \textit{Lagrangian projection} of $\Lambda$.

In order to produce an example of a non-trivial Lagrangian concordance we will use a sequence of elementary cobordisms as defined in \cite{collarable} and \cite{Ekhoka}. A combination of results from \cite{chantraine_conc}, \cite{collarable} and \cite{Ekhoka} implies that the local moves of Figure \ref{fig:localmove} can be realised by Lagrangian cobordisms (the arrows indicate the increasing $\mathbb{R}$ direction in $\mathbb{R}\times\mathbb{R}^3$). 

\begin{figure}[!h]
\labellist
\small\hair 2pt
\pinlabel {$1$} [bl] at 174 620
\pinlabel {$1'$} [bl] at 398 620
\pinlabel {$2$} [bl] at 113 491
\pinlabel {$2'$} [bl] at 442 490
\pinlabel {$3$} [bl] at 288 347
\pinlabel {$4$} [bl] at 288 179
\pinlabel {$5$} [bl] at 288 36
\pinlabel {$\emptyset$} [bl] at  88 15
\endlabellist
  \centering
  \includegraphics[height=7cm]{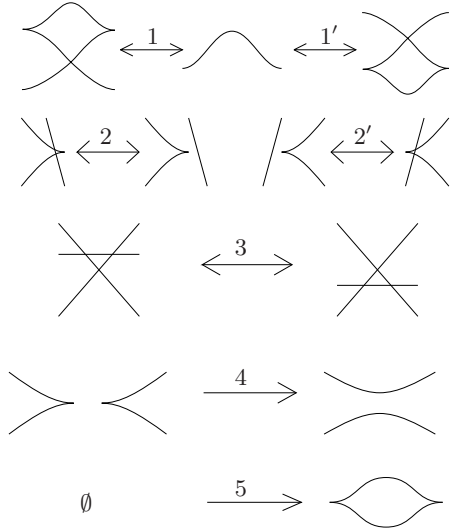}
  \caption{Local bifurcations of fronts along elementary Lagrangian cobordisms.}
  \label{fig:localmove}
\end{figure}

The first three moves are Legendrian Reidemeister moves arising along generic Legendrian isotopies, in each case the associated cobordism is a concordance. The fourth move is a saddle cobordism which corresponds to a $1$-handle attachment. The cobordism corresponding to the fifth move is a disk.

In Section \ref{sec:legendr-cont-homol} we will compute the induced map in Legendrian contact homology by a concordance. It will then be convenient to have a description of this concordance in terms of the Lagrangian projection. As it is easier in general to draw isotopy of front projections, we will use procedure of \cite{Ngcomputable} to draw Lagrangian projections from front projections.

The idea is to write front projections in piecewise linear forms where the slope of a strand is always bigger than the one under it except before a crossing or a cusp. Such front diagrams are then easily translated into Lagrangian projections. 

In Figure \ref{fig:lagsimplemove} we provide, on the left, the elementary moves in front diagrams of this form associated to elementary cobordisms which we translate then, on the right,  in terms of Lagrangian projections. As in Figure \ref{fig:localmove} the arrows represent the increasing time direction. 

\begin{figure}[!h]
\labellist
\small\hair 2pt
\pinlabel {$\mathbf{II^{-1}}$} [bl] at 450 754
\pinlabel {$\mathbf{II}$} [tl] at 450 724
\pinlabel {$\mathbf{II^{-1}}$} [bl] at 450 635
\pinlabel {$\mathbf{II}$} [tl] at 450 605
\pinlabel {$\mathbf{II^{-1}}$} [bl] at 450 509
\pinlabel {$\mathbf{II}$} [tl] at 450 479
\pinlabel {$\mathbf{II^{-1}\circ III'}$} [bl] at 410 370
\pinlabel {$\mathbf{III'\circ II}$} [tl] at 410 332
\pinlabel {$\mathbf{III'}$} [bl] at 450 251
\pinlabel {$\mathbf{III'}$} [tl] at 450 214
\pinlabel {$\mathbf{II^{-1}\circ IV}$} [bl] at 410 156
\pinlabel {$\mathbf{V}$} [bl] at 450 27
\pinlabel {$\emptyset$} [bl] at 360 15
\pinlabel {$\emptyset$} [bl] at 50 15
\endlabellist
  \centering
  \includegraphics[height=10cm]{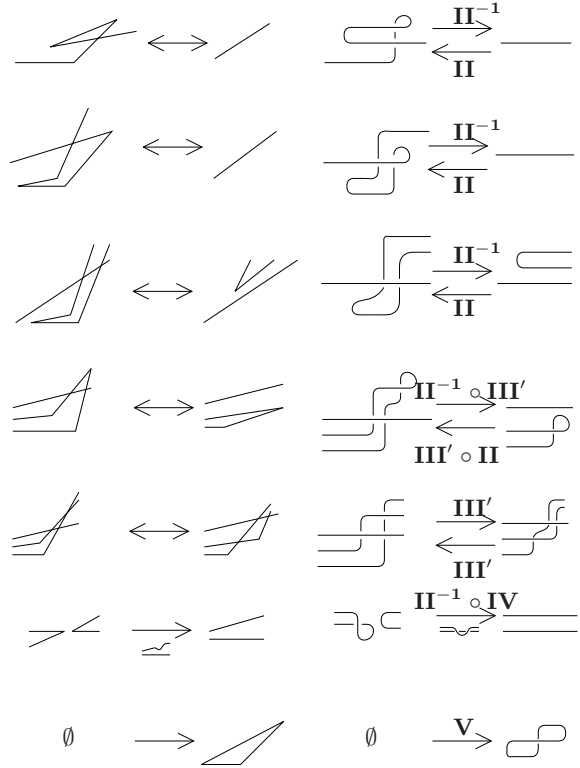}
  \caption{Lagrangian projections of elementary cobordisms.}
  \label{fig:lagsimplemove}
\end{figure}

We label an arrow according to the corresponding bifurcation of the Lagrangian projection where $\mathbf{II}$, $\mathbf{III}$ and $\mathbf{III'}$ correspond to the notation of \cite{Kalman_One_parameter_family}. However, as a cobordism from $\Lambda^-$ to $\Lambda^+$ induce a map from $\mathcal{A}(\Lambda^+)$ to $\mathcal{A}(\Lambda^-)$ (i.e. following the decreasing time direction) we labelled a move in Figure \ref{fig:lagsimplemove} by the corresponding move from \cite{Kalman_One_parameter_family} following the arrow backward. As an example, if $\Lambda^-$ differs from $\Lambda^+$ by a move number $\mathbf{II}$ from \cite{Kalman_One_parameter_family} we will label the arrow by a $\mathbf{II}^{-1}$ as it is this move we will use to compute the map from $\mathcal{A}(\Lambda^+)$ to $\mathcal{A}(\Lambda^-)$. We denote by $\mathbf{IV}$ the saddle cobordism denoted $L_{sa}$ in \cite{Ekhoka} and by $\mathbf{V}$ the Lagrangian filling of $\Lambda_0$ denoted by $L_{mi}$ in \cite{Ekhoka}. In move number $4$, we also provide an intermediate step which corresponds to the creation of two Reeb chords one of which being then resolved by the cobordism (this procedure guaranties that the smallest newly created chord is contractible). 

This language being understood we will be able to translate any bifurcation of fronts as a bifurcation of Lagrangian projections and we will keep drawing qualitative Lagrangian projections.

\section{Example of a non-trivial concordance.}
\label{sec:non-triv-conc}

Using the moves of Figure \ref{fig:localmove} we are able to provide a non trivial Lagrangian concordance from $\Lambda_0$ to $\Lambda$. Note that the knot $m(9_{46})$ is the first Legendrian knot in the Legendrian knot atlas of \cite{atlas} with $$g_s(K)=0 \text{ and } \max\{tb(\Lambda)\vert \Lambda\text{ Legendrian representative of } K\}=-1,$$ thus, following \cite[Theorem 1.4]{chantraine_conc}, it is the simplest candidate for such an example. The bifurcations of the fronts along the non trivial concordance is given on Figure \ref{fig:notriv}.

\begin{figure}[hftp]
\labellist
\small\hair 2pt
\pinlabel{$4$} [bl] at 275 436
\pinlabel{$2\circ 2'\circ 2'$} [bl] at 180 339
\pinlabel{$1\circ 1$} [bl] at 265 272
\pinlabel{$2'$} [bl] at 245 166
\pinlabel{$5$} [bl] at 275 84
\endlabellist
  \centering
  \includegraphics[height=7cm]{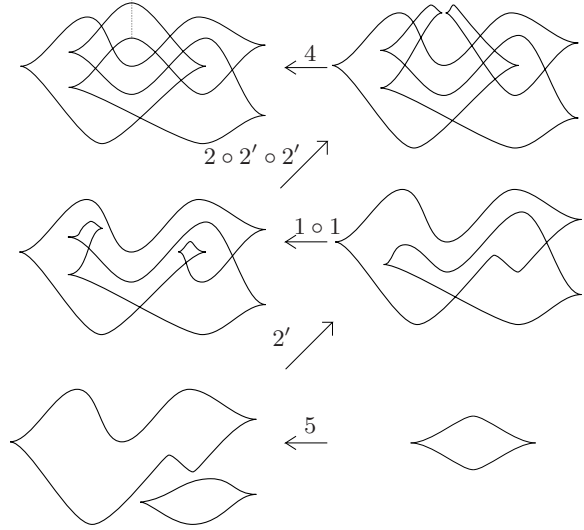}
  \caption{A non trivial Lagrangian concordance.}\label{fig:notriv}
\end{figure}

One can see that it is indeed a concordance either by using \cite[Theorem 1.3]{chantraine_conc} and deduce from $tb(\Lambda)-tb(\Lambda_0)=0$ that the genus of the cobordism is $0$ or by explictly seeing that the projection to $\mathbb{R}$ of $C$ has only two critical points, one of index $1$ and one of index $0$ which implies that $C$ is a cylinder.

\section{Legendrian contact homology of $\Lambda$ and some geometrical maps.}
\label{sec:legendr-cont-homol}

We compute now the boundary operator on the Chekanov algebra of $\Lambda$ (see \cite{Chekanov_DGA_Legendrian}). As $r(\Lambda)=0$ it is a differential $\mathbb{Z}$-graded algebra over $\mathbb{Z}_2$ freely generated by the double points of the Lagrangian projection of $\Lambda$. The generators of $\mathcal{A}(\Lambda)$ are represented on Figure \ref{fig:LCHlag} where each $a_i$ has degree $1$, each $b_i$ degree $0$ and each $c_i$ degree $-1$.

\begin{figure}[hftp]
\labellist
\small\hair 2pt
\pinlabel {$a_1$} [bl] at 515 173
\pinlabel {$a_2$} [bl] at 531 76
\pinlabel {$a_3$} [bl] at 397 144
\pinlabel {$b_6$} [bl] at 482 150
\pinlabel {$a_4$} [bl] at 366 175
\pinlabel {$b_5$} [bl] at 362 115
\pinlabel {$b_3$} [bl] at 338 199
\pinlabel {$b_4$} [bl] at 326 145
\pinlabel {$c_1$} [bl] at 303 168
\pinlabel {$c_2$} [bl] at 258 40
\pinlabel {$a_5$} [bl] at 207 180
\pinlabel {$b_1$} [bl] at 158 230
\pinlabel {$b_2$} [bl] at 116 138
\endlabellist  
  \centering
  \includegraphics[height=4cm]{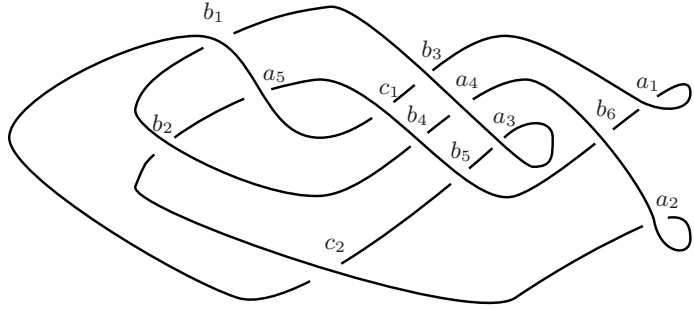}
  \caption{Generators of $\mathcal{A}(\Lambda)$.}
  \label{fig:LCHlag}
\end{figure}

The boundary operator on generators counts degree one immersed polygons with one positive corner and several negative corners and in our situation gives:
\begin{align*}
  \partial a_1&=1+a_5c_2b_2+b_1b_6+b_2\\
\partial a_2&=1+b_2c_2a_4b_2+b_2c_2b_3a_5+b_6b_4b_2+b_6c_1a_5+b_6+b_2\\
\partial a_3&=1+a_4b_2c_2+b_3a_5c_2+b_3+b_2b_5\\
\partial a_4&=1+b_3b_1+b_2b_4\\
\partial a_5&=b_1b_2\\
\partial b_1=\partial b_2&=0\\
\partial b_3&=b_2c_1\\
\partial b_4&=c_1b_1\\
\partial b_5&=b_4b_2c_2+c_1a_5c_2+c_2+c_1\\
\partial b_6&=b_2c_2b_2\\
\partial c_1=\partial c_2&=0.
\end{align*}

It is then extended to the whole algebra by Leibniz' rule: $\partial(ab)=\partial(a)b+a\partial(b)$.

We will now compute the map between Chekanov algebras  associated to the concordance $C$ of Figure \ref{fig:notriv}. At each step we use the results of \cite{Ekhoka} which give a combinatorial description of the map associated to each elementary cobordism. 

\begin{figure}[hftp]
\labellist
\small\hair 2pt
\pinlabel{$C_1$} [bl] at 285 745
\pinlabel{$C_2$} [br] at 290 650
\pinlabel{$C_3$} [bl] at 285 540
\pinlabel{$C_4$} [br] at 290 410
\pinlabel{$C_5$} [bl] at 285 311
\pinlabel{$C_6$} [br] at 290 181
\pinlabel{$C_7$} [bl] at 325 70
\endlabellist
  \centering
  \includegraphics[height=9cm]{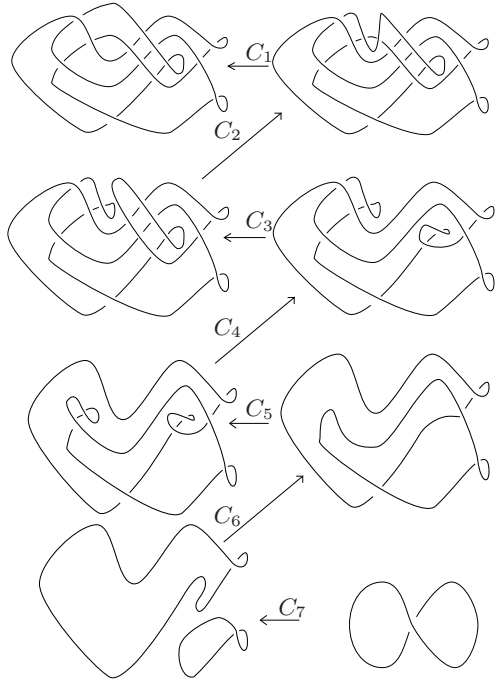}
  \caption{Bifurcations of Lagrangian projections along the non-trivial concordance.}\label{fig:notrivlag}
\end{figure}

On Figure \ref{fig:notrivlag} we see the bifurcations of the Lagrangian projections along $C$ using the correspondence between front moves and Lagrangian moves of Figure \ref{fig:lagsimplemove}, for convenience we split the first two steps in two steps each. For a cobordism $C_i$ we denote the differential of the DGA associated to the upper level by $\partial^+_{C_i}$ and the one corresponding to the lower level by $\partial^-_{C_i}$ (of course $\partial^+_{C_{i+1}}=\partial^-_{C_i}$).  At each step we compute the map associated to these moves between the corresponding Chekanov algebras heavily using the results of \cite[Section 6]{Ekhoka}. We provide the precise section of this paper we use for each of the corresponding move. We decorate the labels of the bifurcations of the Lagrangian projections with subscripts precising the chords involved by each move.

\subsection{Map associated to $C_1$.}
\label{sec:map-associated-c_1}

The bifurcation associated to the cobordism $C_1$ is $\mathbf{II_{ab}}$ as in Figure \ref{fig:aug11}. The computation of the map associated to this move is the most involved of all the DGA maps described in \cite{Kalman_One_parameter_family} and \cite{Ekhoka}.

\begin{figure}[!h]
\labellist
\small\hair 2pt
\pinlabel {$b_5$} [bl] at 325 139
\pinlabel {$b_4$} [bl] at 300 176
\pinlabel {$a_1$} [bl] at 473 215
\pinlabel {$a_2$} [bl] at 485 95
\pinlabel {$b_6$} [bl] at 435 177
\pinlabel {$a_3$} [bl] at 362 171
\pinlabel {$a_4$} [bl] at  334 210
\pinlabel {$b_3$} [bl] at 306 240
\pinlabel {$c_1$} [bl] at 265 201
\pinlabel {$a_5$} [bl] at 150 205
\pinlabel {$c_2$} [bl] at 230 45
\pinlabel {$b_1$} [bl] at 146 266
\pinlabel {$b_2$} [bl] at 107 171 
\pinlabel {$b_5$} [bl] at 945 139
\pinlabel {$b_4$} [bl] at 920 176
\pinlabel {$a_1$} [bl] at 1098 208
\pinlabel {$a_2$} [bl] at 1105 95
\pinlabel {$b_6$} [bl] at 1060 177
\pinlabel {$a_3$} [bl] at 987 171
\pinlabel {$a_4$} [bl] at  959 210
\pinlabel {$b_3$} [bl] at 931 240
\pinlabel {$c_1$} [bl] at 890 201
\pinlabel {$a_5$} [bl] at 770 205
\pinlabel {$c_2$} [bl] at 855 45
\pinlabel {$b_1$} [bl] at 771 266
\pinlabel {$b_2$} [bl] at 732 171 
\pinlabel {$a$} [bl] at 815 215
\pinlabel {$b$} [bl] at 855 222
\endlabellist  
\centering
  \includegraphics[width=12cm]{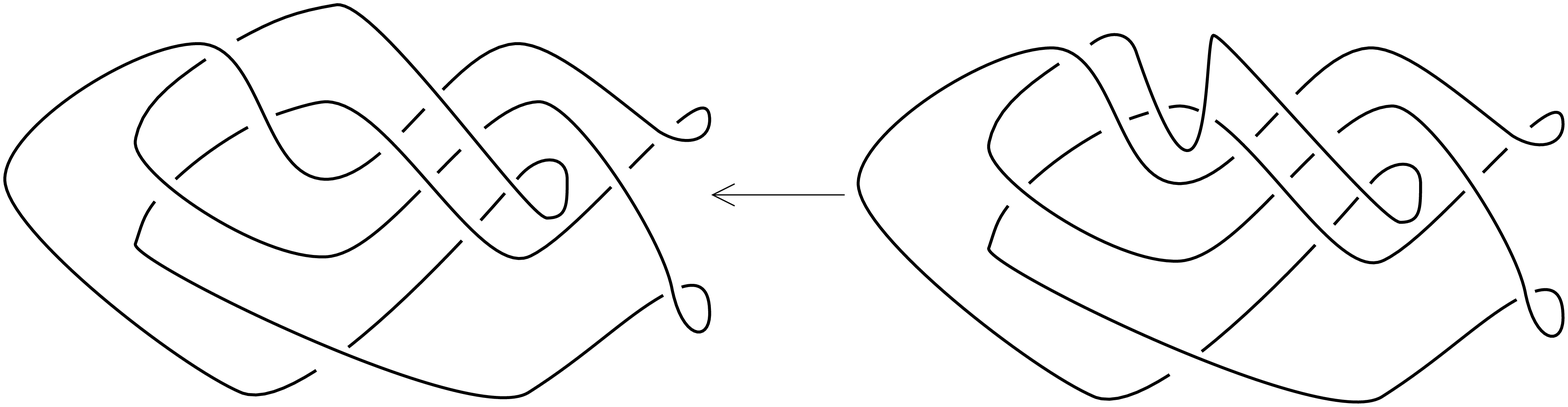}
        \caption{$\mathbf{II_{ab}}$.}
\label{fig:aug11}
\end{figure}

Following \cite[Section 6.3.4]{Ekhoka}, in order to compute $\varphi_{C_1}$ we need first to know $\partial_{C_1}^-$. We have:
\begin{align*}
  \partial_{C_1}^-a_1&=1+a_5c_2b+b_1b_6+b\\
\partial_{C_1}^-a_2&=1+b_2c_2a_4b_2+b_2c_2b_3a_5+b_6b_4b_2+b_6c_1a_5+b_6c_2a+b_6+b_2\\
\partial_{C_1}^-a_3&=1+a_4b_2c_2+b_3a_5c_2+b_3+bb_5+ac_2\\
\partial_{C_1}^-a_4&=1+b_3b_1+bb_4\\
\partial_{C_1}^-a_5&=b_1b_2\\
\partial_{C_1}^-b_1=\partial_{C_1}^-b_2&=0\\
\partial_{C_1}^-b_3&=bc_1\\
\partial_{C_1}^-b_4&=c_1b_1\\
\partial_{C_1}^-b_5&=b_4b_2c_2+c_1a_5c_2+c_2+c_1\\
\partial_{C_1}^-b_6&=b_2c_2b\\
\partial_{C_1}^-c_1=\partial_{C_1}^-c_2&=0\\
\partial_{C_1}^-a&=b+b_2\\
\partial_{C_1}^-b&=0.
\end{align*}

Which we compare to $\partial_{C_1}^+$ computed above which gave
\begin{align*}
  \partial^+_{C_1} a_1&=1+a_5c_2b_2+b_1b_6+b_2\\
\partial^+_{C_1} a_2&=1+b_2c_2a_4b_2+b_2c_2b_3a_5+b_6b_4b_2+b_6c_1a_5+b_6+b_2\\
\partial^+_{C_1} a_3&=1+a_4b_2c_2+b_3a_5c_2+b_3+b_2b_5\\
\partial^+_{C_1} a_4&=1+b_3b_1+b_2b_4\\
\partial^+_{C_1} a_5&=b_1b_2\\
\partial^+_{C_1} b_1=\partial^+_{C_1} b_2&=0\\
\partial^+_{C_1} b_3&=b_2c_1\\
\partial^+_{C_1} b_4&=c_1b_1\\
\partial^+_{C_1} b_5&=b_4b_2c_2+c_1a_5c_2+c_2+c_1\\
\partial^+_{C_1} b_6&=b_2c_2b_2\\
\partial^+_{C_1} c_1=\partial^+_{C_1} c_2&=0.
\end{align*}

A priori, in order to compute the associated map $\varphi_{C_1}$ we need to order the Reeb chord according to the length filtration (see \cite[Section 3.1]{Kalman_One_parameter_family} and \cite[Section 6.3.4]{Ekhoka}). This ensure that when computing $\varphi_{C_1}(a)$ we already know the image by $\varphi_{C_1}$ of any letter appearing in $\partial^-_{C_1}(a)$. But we actually do not need to understand the whole filtration in a concrete example. For this note that for any generator $d$ of $\mathcal{A}(\Lambda^+)$ if $b$ is not a letter appearing in $\partial^-_{C_1}(d)$ then $\varphi_{C_1}(d)=d$ regardless of its action.  Thus in the end we need to understand the filtration on $a_1$, $a_3$, $a_4$, $b_3$ and $b_6$. One easily see that the action of $a_1$ can be made as big as we want without changing any other action. Then from the fact that $\partial^\pm$ decreases the action one get that $h(a_1)>h(a_3)>h(a_4)>h(b_3)$ and that $h(a_1)>h(b_6)$. This is enough to proceed with inductive process (as $b_6$ only appears in $\partial(a_1)$ we treat it as having action greater than $a_3$).

Also  note that $\partial^-_{C_1}(a)=b=b+0$ which give $v=0$ (following the notation from \cite{Ekhoka}). 

We start with $b_3$ following the notation of \cite[Section 6.3.4]{Ekhoka} we need to write $\partial^-_{C_1}b_3=\sum B_1bB_2b\ldots B_kbA$ where all $B's$ are words with letters in the generator of $\mathcal{A}(\Lambda^+)$ (with lower action than $b_3$) and where every occurence of $b$ in $A$ follows an occurence of $a$. In our situation we have $\partial_{C_1}^-b_3=bc_1=bA$ with $A=c_1$ (and we have no word of type $B_i$). Thus $b_3$ is mapped to $b_3+aA=b_3+ac_1$. 

We then proceed for $a_4$, we get $\partial_{C_1}^-a_4=1+b_3b_1+bb_4=A_1+A_2+bA_3$ with $A_1=1$, $A_2=b_3b_1$ and $A_3=b_4$ (again no $B$'s). Only $A_3$ is of interest here (as it belongs to a monomial containing $b$) and implies that $a_4$ is mapped to $a_4+ab_4$.

For $a_3$ we have $\partial_{C_1}^-a_3=1+a_4b_2c_2+b_3a_5c_2+b_3+bb_5+ac_2$. The only relevant monomial is $bb_5$ implying that $a_3$ is mapped to $a_3+ab_5$.

As for $b_6$ we have $\partial_{C_1}^-b_6=b_2c_2b=Bb$. This implies that $b_6$ is mapped to $b_6+\varphi_{C_1}(B)a=b_6+b_2c_2a$.

Finally for $a_1$ we have $\partial_{C_1}^-a_1=1+a_5c_2b+b_1b_6+b= A_1+A_2b+A_3+A_4b$ with the only relevant $A_i$ being $A_2=a_5c_2$ and $A_4=1$ giving that $a_1$ is mapped to $a_1+a_5c_2a+a$.

In summary we have that $\varphi_{C_1}$ does the following:
\begin{align*}
a_1&\rightarrow a_1+a+a_5c_2a\\
a_3&\rightarrow a_3+ab_5\\
a_4&\rightarrow a_4+ab_4\\
  b_3&\rightarrow b_3+ac_1\\
b_6&\rightarrow b_6+b_2c_2a\\
\end{align*}

and all other generators are mapped to themselves.

\subsection{Map associated to $C_2$.}
\label{sec:map-associated-c_1-1}

The bifurcation associated to $C_2$ is of type $\mathbf{IV_b}$ using the notations of Figure \ref{fig:saddleaug1}.
\begin{figure}[!h]
\labellist
\small\hair 2pt
\pinlabel {$b_5$} [bl] at 331 139
\pinlabel {$b_4$} [bl] at 300 176
\pinlabel {$a_1$} [bl] at 473 215
\pinlabel {$a_2$} [bl] at 485 95
\pinlabel {$b_6$} [bl] at 435 177
\pinlabel {$a_3$} [bl] at 362 171
\pinlabel {$a_4$} [bl] at  334 210
\pinlabel {$b_3$} [bl] at 306 240
\pinlabel {$c_1$} [bl] at 271 201
\pinlabel {$a_5$} [bl] at 150 205
\pinlabel {$c_2$} [bl] at 230 45
\pinlabel {$b_1$} [bl] at 146 266
\pinlabel {$b_2$} [bl] at 112 171 
\pinlabel {$a$} [bl] at 190 215
\pinlabel {$b$} [bl] at 230 222
\pinlabel {$b_5$} [bl] at 950 139
\pinlabel {$b_4$} [bl] at 920 176
\pinlabel {$a_1$} [bl] at 1098 215
\pinlabel {$a_2$} [bl] at 1110 95
\pinlabel {$b_6$} [bl] at 1060 177
\pinlabel {$a_3$} [bl] at 987 171
\pinlabel {$a_4$} [bl] at  959 210
\pinlabel {$b_3$} [bl] at 931 240
\pinlabel {$c_1$} [bl] at 891 201
\pinlabel {$a_5$} [bl] at 770 205
\pinlabel {$c_2$} [bl] at 855 45
\pinlabel {$b_1$} [bl] at 771 266
\pinlabel {$b_2$} [bl] at 737 171 
\pinlabel {$a$} [bl] at 810 215
\endlabellist  
 \centering
  \includegraphics[width=12cm]{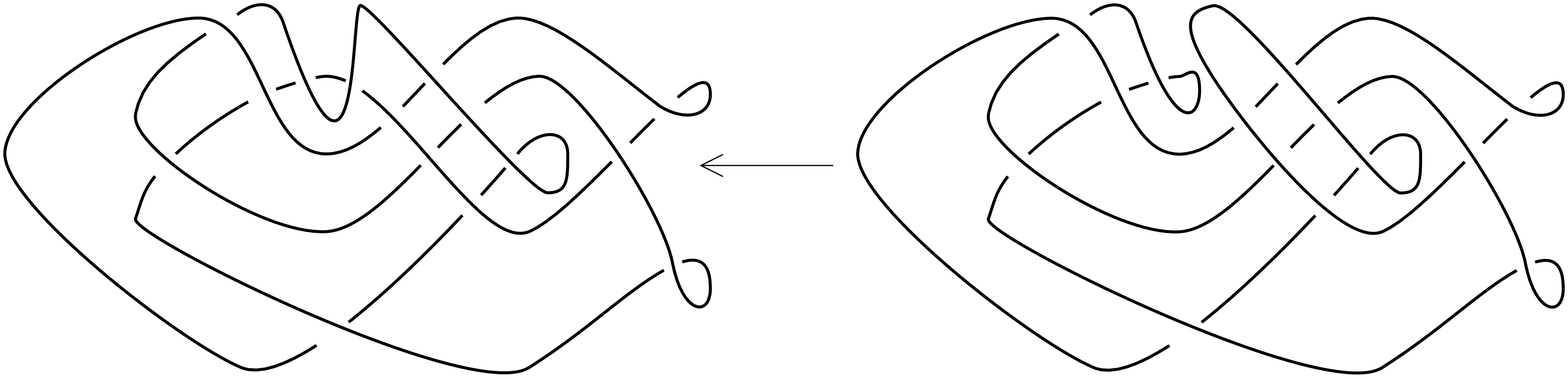}
  \caption{Saddle cobordism $\mathbf{IV_b}$.}
  \label{fig:saddleaug1}
\end{figure}

An easy verification shows that the contractible Reeb chord $b$ is simple (in the sense of \cite{Ekhoka}). We can thus apply \cite[Proposition 6.17]{Ekhoka} and count immersed polygons with two positive corners (one on $b$). We get only three of those (the $\pm$ superscripts design postive and negative corners of the polygons):
\begin{align*}
  a_2^+b_6^-b^+a^-\\
b_4^+b^+\\
b_5^+b^+a^-c_2^-.\\
\end{align*}

Which gives that the map $\varphi_{C_2}$ does the following:

\begin{align*}
  a_2&\rightarrow a_2+b_6a\\
b_4&\rightarrow b_4+1\\
b_5&\rightarrow b_5+ac_2\\
b&\rightarrow 1
\end{align*}

all other generators being mapped to themselves.

This changes the differential as follows:
\begin{align*}
 \partial_{C_2}^-a_1&=a_5c_2+b_1b_6\\
\partial_{C_2}^-a_2&=1+b_2c_2a_4b_2+b_2c_2b_3a_5+b_6b_4b_2+b_6c_1a_5+b_2\\
\partial_{C_2}^-a_3&=1+a_4b_2c_2+b_3a_5c_2+b_3+b_5\\
\partial_{C_2}^-a_4&=b_3b_1+b_4\\
\partial_{C_2}^-a_5&=b_1b_2\\
\partial_{C_2}^-b_1=\partial_{C_2}^-b_2&=0\\
\partial_{C_2}^-b_3&=c_1\\
\partial_{C_2}^-b_4&=c_1b_1\\
\partial_{C_2}^-b_5&=b_4b_2c_2+c_1a_5c_2+c_1\\
\partial_{C_2}^-b_6&=b_2c_2\\
\partial_{C_2}^-c_1=\partial_{C_2}^-c_2&=0\\
\partial_{C_2}^-a&=1+b_2\\
\partial_{C_2}^-b&=0.
\end{align*}

\subsection{Map associated to $C_3$.}
\label{sec:map-associated-c_2}

Using the notation of Figure \ref{fig:aug1_3}, the bifurcations associated to $C_3$ are given by first $\mathbf{II^{-1}_{b_3 c_1}}$ then $\mathbf{II^{-1}_{a_4 b_4}}$ (going in the decreasing $t$ direction).

\begin{figure}[!h]
\labellist
\small\hair 2pt
\pinlabel {$b_5$} [bl] at 325 155
\pinlabel {$b_4$} [bl] at 295 190
\pinlabel {$a_1$} [bl] at 470 225
\pinlabel {$a_2$} [bl] at 480 110
\pinlabel {$b_6$} [bl] at 430 190
\pinlabel {$a_3$} [bl] at 355 185
\pinlabel {$a_4$} [bl] at  330 220
\pinlabel {$b_3$} [bl] at 300 250
\pinlabel {$c_1$} [bl] at 265 215
\pinlabel {$a_5$} [bl] at 145 220
\pinlabel {$c_2$} [bl] at 225 60
\pinlabel {$b_1$} [bl] at 140 280
\pinlabel {$b_2$} [bl] at 105 185 
\pinlabel {$a$} [bl] at 180 230
\pinlabel {$a_1$} [bl] at 1080 225
\pinlabel {$a_2$} [bl] at 1090 100
\pinlabel {$b_6$} [bl] at 1040 190
\pinlabel {$a_5$} [bl] at 760 220
\pinlabel {$c_2$} [bl] at 845 60
\pinlabel {$b_1$} [bl] at 765 285
\pinlabel {$b_2$} [bl] at 725 185 
\pinlabel {$a$} [bl] at 805 230
\pinlabel {$b_5$} [tr] at 941 155
\pinlabel {$a_3$} [bl] at 979 185
\endlabellist  
 \centering
  \includegraphics[width=12cm]{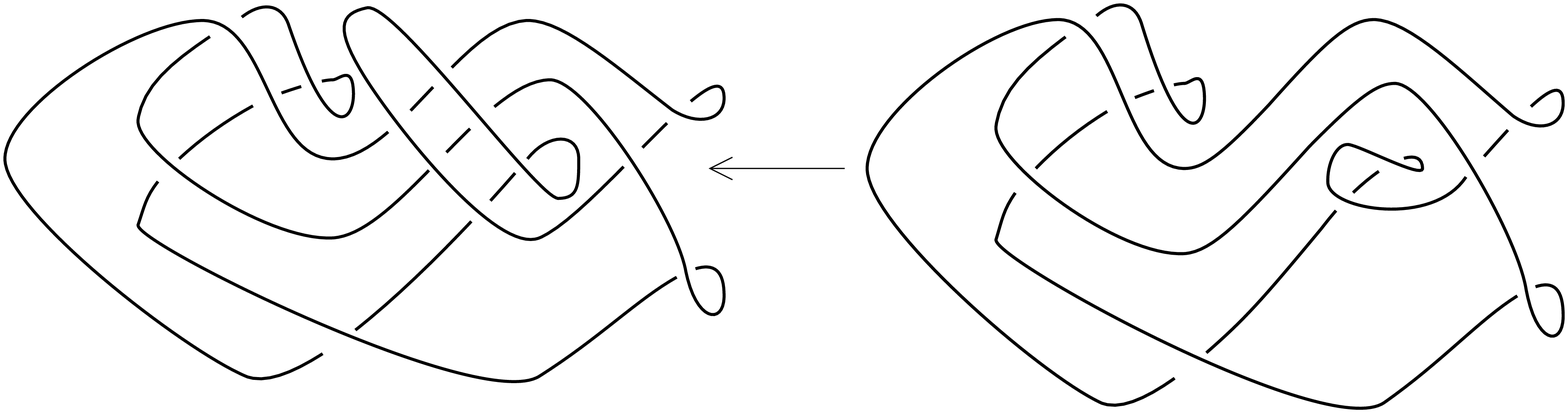}
  \caption{$\mathbf{II^{-1}_{b_3 c_1}\circ  II^{-1}_{a_4 b_4}}$.}
  \label{fig:aug1_3}
\end{figure}

From $\partial^+_{C_3}(b_3)=c_1=c_1+v$ with $v=0$ we deduce (following \cite[Section 6.3.3]{Ekhoka}) that at the first bifurcation $b_3$ maps to $0$ and $c_1$ maps to $v$ thus to $0$. This implies that in the middle of the cobordism one has $\partial(a_4)=b_4$ implying that $a_4$ and $b_4$ maps to $0$. Thus $\varphi_{C_3}$ does the following:
\begin{align*}
  b_3&\rightarrow 0\\
c_1&\rightarrow 0\\
a_4&\rightarrow 0\\
b_4&\rightarrow 0\\
\end{align*}

all other generators being mapped to themselves.
\subsection{Map associated to $C_4$.}
\label{sec:map-associated-c_2-1}

Following the notation of Figure \ref{fig:aug1_4}, the bifurcations associated to the cobordism $C_4$ are, again following the decreasing $t$ direction, first  $\mathbf{III'_{b_1 a_5 a}}$ then $\mathbf{II^{-1}_{a_5 b_1}}$.
\label{sec:ii-1a_5-b_1circ}
\begin{figure}[!h]
\labellist
\small\hair 2pt
\pinlabel {$a_3$} [bl] at 380 180
\pinlabel {$b_5$} [tr] at 338 155
\pinlabel {$a_1$} [bl] at 473 220
\pinlabel {$a_2$} [bl] at 485 100
\pinlabel {$b_6$} [bl] at 435 180
\pinlabel {$a_5$} [bl] at 150 210
\pinlabel {$c_2$} [bl] at 230 50
\pinlabel {$b_1$} [bl] at 146 280
\pinlabel {$b_2$} [bl] at 112 175 
\pinlabel {$a$} [bl] at 190 220
\pinlabel {$a_1$} [bl] at 1090 220
\pinlabel {$a_2$} [bl] at 1095 100
\pinlabel {$b_6$} [bl] at 1045 180
\pinlabel {$c_2$} [bl] at 845 50
\pinlabel {$a$} [bl] at 748 185
\pinlabel {$b_2$} [bl] at 695 150
\pinlabel {$a_3$} [bl] at 990 180
\pinlabel {$b_5$} [tr] at 955 155
\endlabellist  
 \centering
  \includegraphics[width=12cm]{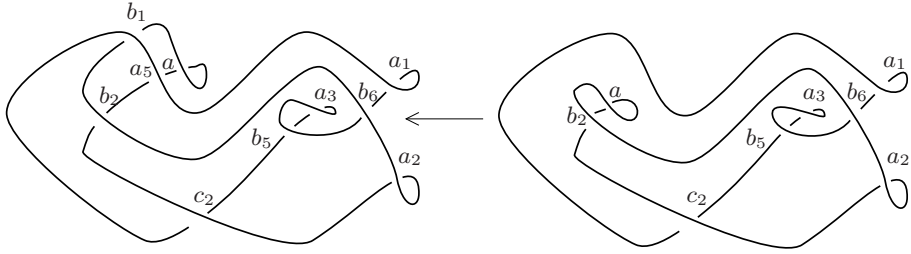}
  \caption{$\mathbf{III'_{b_1 a_5 a}\circ II^{-1}_{a_5 b_1}}$.}
  \label{fig:aug1_4}
\end{figure}

To compute the map associated to $\mathbf{III'_{b_1 a_5 a}}$ we apply \cite[Section 6.3.2]{Ekhoka} and get that $a_5$ maps to $a_5+b_1a$ and all other generators are mapped to themselves.

One computes that  in the middle $\partial(a_5)=\partial^+_{C_4}(a_5)+\partial(b_1a)=b_1b_2+b_1+b_1b_2=b_1$. Applying again \cite[Section 6.3.3]{Ekhoka} we deduce that the bifurcation $\mathbf{II^{-1}_{a_5b_1}}$ maps $a_5$ and $b_1$ to $0$. This implies that $\varphi_{C_4}$ does the following: 
\begin{align*}
  a_5&\rightarrow 0\\
b_1&\rightarrow 0\\
a&\rightarrow a\\
\end{align*}

all other generator being mapped to themselves.

The differential at this step is:
\begin{align*}
  \partial_{C_4}^-a_1&= 0\\
\partial_{C_4}^-a_2&= 1+ b_2\\
\partial_{C_4}^-a&= 1+b_2\\
\partial_{C_4}^-b_2&= 0\\
\partial_{C_4}^-b_6&= b_2c_2\\
\partial_{C_4}^-a_3&= 1+b_5\\
\partial_{C_4}^-b_5&= 0.
\end{align*}

\subsection{Map associated to $C_5$.}
\label{sec:map-associated-c_3}

Using the notation of Figure \ref{fig:augc_3}, the bifurcations corresponding to $C_5$ are $\mathbf{II^{-1}_{a_3b_5}}$ and $\mathbf{II^{-1}_{ab_2}}$ (these are commutative).

\begin{figure}[!h]
\labellist
\small\hair 2pt
\pinlabel {$a_1$} [bl] at 473 215
\pinlabel {$a_2$} [bl] at 485 95
\pinlabel {$b_6$} [bl] at 435 165
\pinlabel {$c_2$} [bl] at 230 45
\pinlabel {$a_1$} [bl] at 1115 205
\pinlabel {$a_2$} [bl] at 1095 95
\pinlabel {$a$} [bl] at 133 182
\pinlabel {$b_2$} [bl] at 75 145
\pinlabel {$a_3$} [bl] at 370 180
\pinlabel {$b_5$} [tr] at 341 155
\pinlabel {$b_6$} [bl] at 1045 190
\pinlabel {$c_2$} [bl] at 845 56
\endlabellist  
 \centering
  \includegraphics[width=12cm]{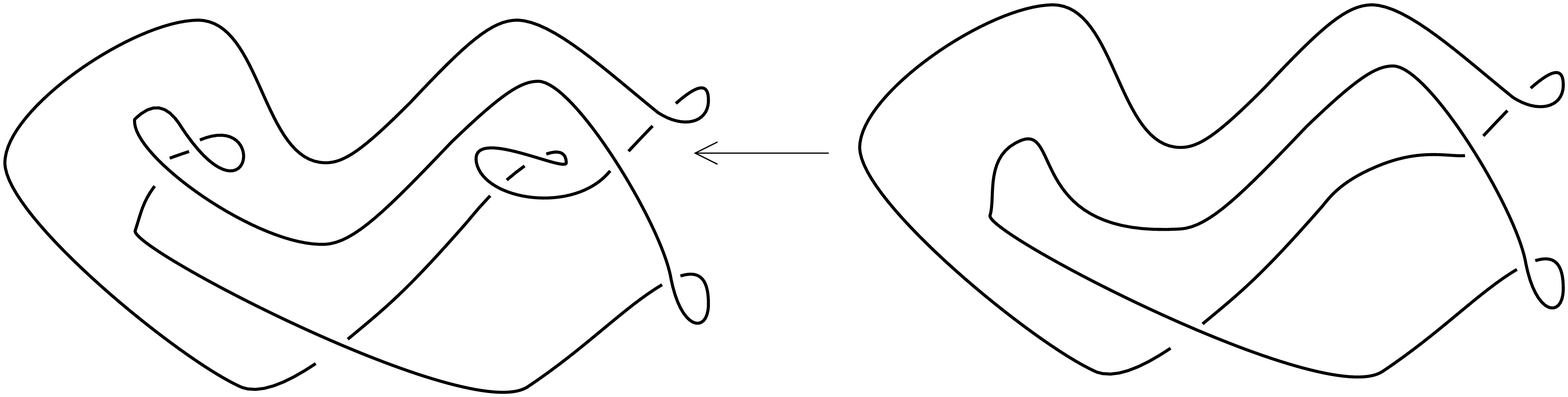}
  \caption{$\mathbf{II^{-1}_{a_3b_5}\circ II^{-1}_{ab_2}}$.}
  \label{fig:augc_3}
\end{figure}

 One easily see that $\varphi_{C_5}$ does the following:
\begin{align*}
a\rightarrow 0\\
b_2\rightarrow 1\\
a_3\rightarrow 0\\
b_5\rightarrow 1  
\end{align*}

and all other generators are mapped to themselves.

The differential becomes:
\begin{align*}
 \partial_{C_5}^-a_1&= 0\\
\partial_{C_5}^-a_2&= 0\\
\partial_{C_5}^-b_6&= c_2\\
\partial_{C_5}^-c_2&= 0.
\end{align*}

\subsection{Map associated to $C_6$.}
\label{sec:map-associated-c_4}

The bifurcation corresponding to $C_6$ is {$\mathbf{II^{-1}_{b_6 c_2}}$.}
\label{sec:ii}
\begin{figure}[!h]
\labellist
\small\hair 2pt
\pinlabel {$a_1$} [bl] at 473 215
\pinlabel {$a_2$} [bl] at 485 95
\pinlabel {$b_6$} [bl] at 435 177
\pinlabel {$c_2$} [bl] at 230 45
\pinlabel {$a_1$} [bl] at 1098 215
\pinlabel {$a_2$} [bl] at 1095 95
\endlabellist  
 \centering
  \includegraphics[width=12cm]{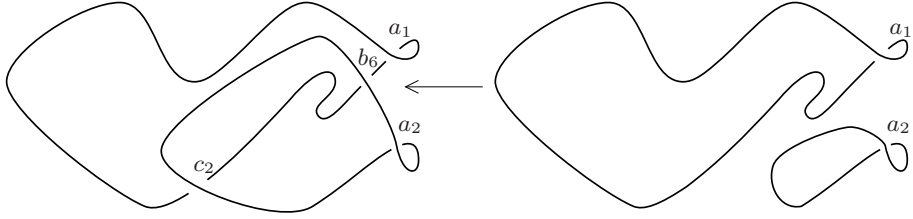}
  \caption{$\mathbf{II^{-1}_{b_6 c_2}}$.}
  \label{fig:aug1_5}
\end{figure}

We have that $\varphi_{C_6}$ does:
\begin{align*}
  a_1&\rightarrow a_1\\
a_2&\rightarrow a_2\\
b_6&\rightarrow 0\\
c_2&\rightarrow 0.
\end{align*}

\subsection{Map associated to $C_7$ and the composition $\mathbf{\varphi_{C}}$.}
\label{sec:comp-varph}
The last part of $C$ is filling one of the components of the link on Figure \ref{fig:aug1_5} with a Lagrangian disk (lets say the one with Reeb chord $a_1$). This has the effect of mapping the corresponding chord to $0$, thus $\varphi_{C_7}(a_1)=0$ and $\varphi_{C_7}(a_2)=a_0$  where $a_0$ is the unique Reeb chords of $\Lambda_0$. 

Combining this to the previous paragraphs we get that the map $$\varphi_{C}=\varphi_{C_7}\circ\varphi_{C_6}\circ\varphi_{C_5}\circ\varphi_{C_4}\circ\varphi_{C_3}\circ\varphi_{C_2}\circ\varphi_{C_1}$$ associated to the concordance of Figure \ref{fig:notriv} is:

\begin{align*}
&a_2\rightarrow a_0\\
&  a_1,a_3,a_4,a_5,b_1,b_3,b_6,c_1,c_2\rightarrow 0\\
&b_2,b_4,b_5\rightarrow 1.
\end{align*}

\section{Lagrangian concordances from $\Lambda_0$ to itself.}
\label{sec:lagr-conc-from}

The aim of this section is to prove the following:
\begin{theorem}
  \label{thm:concmaptriv}
Let $C$ be a Lagrangian concordance from $\Lambda_0$ to $\Lambda_0$ then the map $\varphi_c:\mathcal{A}(\Lambda_0)\rightarrow\mathcal{A}(\Lambda_0)$ induced by $C$ is the identity.
\end{theorem}

This follows from Theorem \ref{thm:lagunknot} of which we give a proof now.

\begin{proof}[Proof of Theorem \ref{thm:lagunknot}]
This is actually a corollary of the main result of \cite{Eliashberg_&_Local_Lagrangian_knots}.

Let $C\subset \mathbb{R}\times S^3$ be an oriented Lagrangian cobordism from $K_0$ to itself. First note that since $tb(K_0)-tb(K_0)=0$ it follows from \cite{chantraine_conc} that $C$ is topologically a cylinder.

The symplectisation of $S^3$ is symplectomorphic to $\mathbb{C}^2\setminus{0}$ with its standard symplectic form. Under this symplectomorphism the $t$-direction becomes the radial direction. A parametrisation of $K_0$ in $S^3$ is given by $\{(cos(\theta),\sin(\theta))\vert\theta\in [0,2\pi)\}\subset \mathbb{C}^2$ i.e. $\Lambda_0=\mathbb{R}^2\cap S^3\subset\mathbb{C}^2$ where $\mathbb{R}^2=\{(x,y)\vert x,y\in\mathbb{R}\}\subset\mathbb{C}^2$. Thus $C$ is a Lagrangian cylinder which coincides near $0$ and outside a compact ball with the trivial Lagrangian plane, i.e. $C_1=C\cup \{0\}$ is local Lagrangian knot (following the terminology of \cite{Eliashberg_&_Local_Lagrangian_knots}). It follows from the main result of \cite{Eliashberg_&_Local_Lagrangian_knots} that there exist a compactly supported Hamiltonian diffeomorphism $\phi_H$ such that $\phi_H(C_1)=\mathbb{R}^2\subset\mathbb{C}^2$.

For $\epsilon>0$ we denote by $D_\epsilon$ the ball of radius $\epsilon$ in $\mathbb{C}^2$. Take $\epsilon$ sufficiently small so that $C_\epsilon:=C_1\cap D_\epsilon=\mathbb{R}^2\cap D_\epsilon$. Since $\phi_H$ maps $C_1$ to $\mathbb{R}^2$ then $\phi_H(C_\epsilon)\subset \mathbb{R}^2$ and there exists a compactly supported diffeomorphism isotopic to the identity $f$  of $\mathbb{R}^2$ such that $f(\phi_H(C_\epsilon))=C_\epsilon$. Using standard construction one can extend $f$ to a compactly supported Hamiltonian diffeomorphism $\widetilde{f}$ of $\mathbb{C}^2$ (which by assumption preserves $\mathbb{R}^2$). Thus $\phi_1=\widetilde{f}\circ\phi_H$ is a compactly supported Hamiltonian diffeomorphism mapping $C_1$ to $\mathbb{R}^2$ such that $\phi_1\vert_{C_\epsilon}=Id$. Now standard application of Moser's path method leads to an Hamiltonian diffeomorphism $\phi'$ supported in $D_\epsilon$  such that $\phi'$ preserves $\mathbb{R}^2$ and $\phi'\circ\phi_1\vert_{D_{\epsilon'}}=Id$ for $\epsilon'<<\epsilon$. Restricting $\phi'\circ\phi_1$ to $\mathbb{C}^2\setminus\{0\}$ proves the theorem. 
\end{proof}

We are now able to prove Theorem \ref{thm:concmaptriv}.

\begin{proof}[Proof of Theorem \ref{thm:concmaptriv}]
 Take a contact embedding of $(\mathbb{R}^3,\xi_0)\rightarrow (S^3,\xi_0)$ as in \cite[Proposition 2.1.8]{Geiges_Intro} such that $\Lambda_0$ is mapped to $K_0$. This embedding induces a symplectic embedding of $\mathbb{R}\times \mathbb{R}^3$ in $\mathbb{R}\times S^3\simeq \mathbb{C}^2\setminus\{0\}$. Under this identification the concordance $C$ maps to a concordance from $K_0$ to itself. Theorem \ref{thm:lagunknot} implies that there exist a compactly supported symplectomorphism $\phi$ mapping $C$ to the trivial cylinder of $K_0$.

Since $\phi$ is the identity near $\pm\infty$, for any cylindrical almost complex structure $J$ on $\mathbb{R}\times S^3$ admissible (in the sense of \cite{Ekholm_Rational_SFT}) for the trivial concordance we get that $(\phi^{-1})^*J$ is admissible for the original concordance $C$. This implies the induced map by $C$ is the same map as the one induced by $\mathbb{R}\times K_0$ which is the identity (because the only degree $0$ pseudo-holomorphic curve on the trivial concordance is the trivial one). Since $H(\mathcal{A}(\Lambda_0))=\mathcal{A}(\Lambda_0)$ and the induced map in homology by $\varphi_C$ do not depends on auxiliary choices, we get that the map do not depend on the choice of the almost complex structure cylindrical at infinities. This conclude the proof. 
\end{proof}

\section{Non symmetry of Lagrangian concordances.}
\label{sec:non-symetry-th}

In order to prove Theorem \ref{thm:principal} we use the augmentation category of $\Lambda$ denoted by $\mathit{Aug}(\Lambda)$. This is an $\mathcal{A}_\infty$-category defined in \cite{augcat} whose objects are augmentations of the Chekanov algebra and morphisms in the homological category are bilinearised Legendrian contact cohomology groups. 

Recall that an augmentation $\varepsilon$ of a DGA $(\mathcal{A},\partial)$ over $\mathbb{Z}_2$ is simply a DGA map from $(\mathcal{A},\partial)$ to $(\mathbb{Z}_2,0)$. 

Bilinearised cohomology groups are generalisations of linearised Legendrian contact cohomology groups (as defined in \cite{Chekanov_DGA_Legendrian}) introduced in \cite{augcat} using two augmentations instead of one and keeping track of the non-commutativity of $\mathcal{A}(\Lambda)$. Basically for two augmentations $\varepsilon_1$ and $\varepsilon_2$ and a word $b_1\ldots b_k$ in $\partial a$ the expression $$\sum_j\varepsilon_1(b_1)\varepsilon_1(b_2)\ldots \varepsilon_1(b_{j-1})\cdot b_j \cdot \varepsilon_2(b_{j+1})\ldots\varepsilon_2(b_k)$$ contributes to $d^{\varepsilon_1,\varepsilon_2}a$.

Dualising $d^{\varepsilon_1,\varepsilon_2}$ leads to bilinearised Legendrian contact cohomology differential $\mu^1_{\varepsilon_1,\varepsilon_2}:C_{\varepsilon_1,\varepsilon_2}(\Lambda)\rightarrow C_{\varepsilon_1,\varepsilon_2}(\Lambda)$ (where $C_{\varepsilon_1,\varepsilon_2}(\Lambda)$ is the vector space generated by Reeb chords of $\Lambda$) whose homology forms morphisms space in the homological category of the augmentation category. Higher order compositions are defined using similar considerations with more than $2$ augmentations. For instance the composition of morphisms $\mu^2_{\varepsilon_1,\varepsilon_2,\varepsilon_3}$ is defined as the dual of the map $d_2^{\varepsilon_3,\varepsilon_2,\varepsilon_1}$ which to a word $b_1\ldots b_k$ in $\partial a$ associates $$\sum_{i,j}\varepsilon_3(b_1)\ldots\varepsilon_3(b_{i-1})\cdot b_i\cdot \varepsilon_2(b_{i+1})\ldots\varepsilon_2(b_{j-1})\cdot b_j\cdot \varepsilon_1(b_{j+1})\ldots\varepsilon_1(b_k).$$

We are now ready to prove Theorem \ref{thm:principal}.
\begin{proof}[Proof of Theorem \ref{thm:principal}]
  The first part on the existence of the concordance has been proved in Section \ref{sec:non-triv-conc}.  It remains to prove that no concordance from $\Lambda$ to $\Lambda_0$ exists.

Assume that such a concordance $C'$ exists and denote by $\varphi_{C'}:\mathcal{A}(\Lambda_0)\rightarrow\mathcal{A}(\Lambda)$ the induced map. Let $C$ be the concordance of Section \ref{sec:legendr-cont-homol} which induced the map $\varphi_{C}$. 

The concatenation of $C'$ with $C$ leads to a concordance from $\Lambda_0$ to itself. Theorem \ref{thm:concmaptriv} implies that the map induced by this concatenation is $Id:\mathcal{A}(\Lambda_0)\rightarrow\mathcal{A}(\Lambda_0)$. Hence by \cite[Theorem 1.2]{Ekhoka} we get that $\varphi_{C}\circ\varphi_{C'}=Id$. 

Now following \cite[Section 2.4]{augcat} we get that $\varphi_{C'}$ induces an $\mathcal{A}_\infty$-functor $\mathcal{F}_{C'}:\mathit{Aug}(\Lambda)\rightarrow\mathit{Aug}(\Lambda_0)$ (obtained by dualising the components of the map $\varphi_{C'}$). Similarly  $\varphi_{C}$ induces an $\mathcal{A}_\infty$-functor $\mathcal{F}_{C}:\mathit{Aug}(\Lambda_0)\rightarrow\mathit{Aug}(\Lambda)$. From $\varphi_{C}\circ\varphi_{C'}=Id$ we get that $\mathcal{F}_{C'}\circ \mathcal{F}_{C}=Id$. 

Note that $\mathcal{A}(\Lambda_0)$ has only one augmentation $\varepsilon_0$ (which maps $a_0$ to $0$). By definition of $\mathcal{F}_C$ its action on the object of the augmentation category is given by $\varepsilon\rightarrow \varepsilon\circ \varphi_C$, thus the explicit computation of Section \ref{sec:legendr-cont-homol} shows that $\mathcal{F}_{C}(\varepsilon_0)=\varphi_{C}\circ\varepsilon_0=\varepsilon_1$ where $\varepsilon_1$ is the first augmentation of Table \ref{tab:twoaug}. Table \ref{tab:twoaug} also shows another augmentation of $\mathcal{A}(\Lambda)$ we will use to compute bilinearised cohomology groups.
  \begin{table}[hftp]
    \centering
    \begin{tabular}{|c|c|c|c|c|c|c|}
\hline       
& $b_1$ &$b_2$&$b_3$&$b_4$&$b_5$&$b_6$\\
\hline
$\varepsilon_1$ & $0$ &$1$&$0$&$1$&$1$&$0$\\
\hline
$\varepsilon_2$&$1$ &$0$&$1$&$0$&$0$&$1$\\
\hline
    \end{tabular}
    \caption{Two augmentations of $\mathcal{A}(\Lambda)$.}
    \label{tab:twoaug}
  \end{table}

We will now show that the two augmentation $\varepsilon_1$ and $\varepsilon_2$ are not equivalent.

Table \ref{tab:d12} gives the bilinearised differential for all possible pairs out of those two augmentations (as $b_1$ and $b_2$ are always mapped to $0$ we omit them from the table).
  \begin{table}[!h]
    \centering
    \begin{tabular}[!h]{|c|c|c|c|c|c|c|c|c|c|}
\hline   
  & $a_1$ & $a_2$ & $a_3$ & $a_4$ & $a_5$ &  $b_3$ & $b_4$ & $b_5$ & $b_6$\\
\hline
$d^{\varepsilon_1,\varepsilon_1}$  & $b_2$ & $b_2$ & $b_3+b_2+b_5$ & $b_2+b_4$ & $b_1$ & $c_1$ & $0$ & $c_1$ & $c_2$\\
\hline
$d^{\varepsilon_2,\varepsilon_2}$  & $b_1+b_2+b_6$ & $b_2+b_6$ & $b_3$ & $b_3+b_1$ & $b_2$  & $0$ & $c_1$ & $c_2+c_1$ & $0$\\
\hline
$d^{\varepsilon_1,\varepsilon_2}$  & $b_1+b_2$ & $b_6+b_2$ & $b_3+b_5$ & $b_3+b_4$ & $0$ & $c_1$ & $c_1$ & $c_1$ & $0$\\
\hline
$d^{\varepsilon_2,\varepsilon_1}$ & $b_6+b_2$ & $b_6+b_4$ & $b_3+b_2$ & $b_1+b_2$ & $b_1+b_2$ & $0$ & $0$ & $c_2+c_1$ & $0$\\
\hline
    \end{tabular}
    \caption{Bilinearised differentials for $\Lambda$.}
    \label{tab:d12}
  \end{table}
 
Notice that for linearised LCH (the first two lines) there are no non-trivial homology in degree $-1$ whereas for the mixed augmentation there is always a generator of degree $-1$. It follows then from \cite[Theorem 1.4]{augcat} that the two augmentations $\varepsilon_1$ and $\varepsilon_2$ are not equivalent. 

In order to conclude, one must study the compositions in the augmentation category and its homological category, thus we need to consider the bilinearised cohomology groups. From Table \ref{tab:d12} we get that the bilinearised differentials in cohomology are those given in Table \ref{tab:mu12}.

\begin{table}[!h]
  \centering
  \begin{tabular}[!h]{|c|c|c|c|c|c|c|c|c|}
   \hline
 & $b_1$ & $b_2$ & $b_3$ & $b_4$ & $b_5$ & $b_6$ &$c_1$ & $c_2$ \\
\hline
$\mu^1_{\varepsilon_1,\varepsilon_1}$ & $a_5$ & $a_1+a_2+a_3+a_4$ & $a_3$ & $a_4$ & $a_3$ & $0$ &$b_3+b_5$ & $b_6$ \\
\hline
$\mu^1_{\varepsilon_2,\varepsilon_2}$ & $a_1+a_4$ & $a_1+a_2+a_5$ & $a_3+a_4$ & $0$ & $0$ & $a_1+a_2$ &$b_4+b_5$ & $b_5$ \\
\hline
$\mu^1_{\varepsilon_1,\varepsilon_2}$ & $a_4+a_5$ & $a_1+a_3+a_4+a_5$ & $a_3$ & $a_2$ & $0$ & $a_1+a_2$ &$b_5$ & $b_5$ \\
\hline
$\mu^1_{\varepsilon_2,\varepsilon_1}$ & $a_1$ & $a_1+a_2$ & $a_3+a_4$ & $a_4$ & $a_3$ & $0$ &$b_3+b_4+b_5$ & $0$ \\
\hline
  \end{tabular}
  \caption{$\mu^1_{\varepsilon_i,\varepsilon_j}$ on $\Lambda$.}
  \label{tab:mu12}
\end{table}

From Table \ref{tab:mu12} we can see that $LCH_{\varepsilon_1}^1$ has one generator $[a_1]=[a_2]$ (since $a_1+a_2=\mu^1_{\varepsilon_1}(b_2+b_3+b_4)$) and that $LCH_{\varepsilon_1}^0$ has dimension $0$ (since $b_6=\mu^1_{\varepsilon_1}(c_2)$). As $\mathcal{F}_{C'}\circ \mathcal{F}_{C}$ is the identity we get that $H(\mathcal{F}_{C'}^1)\circ H(\mathcal{F}^1_{C}): LCH_{\varepsilon_0}(\Lambda_0)\rightarrow LCH_{\varepsilon_0}(\Lambda_0)$ is the identity. This implies that in the homological category $H(\mathcal{F}^1_C):LCH_{\varepsilon_1}(\Lambda)\rightarrow LCH_{\varepsilon_0}(\Lambda_0)$ is surjective in particular the only generator $[a_2]$ of $LCH_{\varepsilon_1}^1(\Lambda)$ is mapped to $[a_0]$ the generator $LCH^1_{\varepsilon_0}(\Lambda_0)$. 

In order to understand the compositions in the category, we need to compute $\partial_2^{\varepsilon_1,\varepsilon_2,\varepsilon_1}$ which gives
\begin{align}
  \label{eq:2}
  a_1&\rightarrow b_1b_6\nonumber\\
a_2&\rightarrow c_2a_4+c_2a_5+b_6b_4\\
a_3&\rightarrow b_2b_5\nonumber\\
a_4&\rightarrow b_3b_1+b_2b_4\nonumber\\
a_5&\rightarrow b_1b_2\nonumber\\
b_3&\rightarrow b_2c_1\nonumber\\
b_4&\rightarrow c_1b_1\nonumber\\
b_6&\rightarrow c_2b_2+b_2c_2.\nonumber
\end{align}

From Formula \eqref{eq:2} we see that $\mu_{\varepsilon_1,\varepsilon_2,\varepsilon_1}^2(a_5,c_2)=a_2\in C_{\varepsilon_1,\varepsilon_1}(\Lambda)$. As the composition $[x]\circ [y]$ in the homological category is given by $[\mu^2(x,y)]$ we get that $[a_5]\circ [c_2]=[a_2]$. Since $\mathcal{F}_{C'}$ is an $\mathcal{A}_\infty$-functor we get that $H(\mathcal{F_{C'}}^1)$ preserves this composition (see \cite[Section 2.3]{augcat}) thus we have that $0\not= [a_0]=H(\mathcal{F}_{C'}^1)([a_2])=H(\mathcal{F}_{C'}^1)([a_5])\circ H(\mathcal{F}_{C'}^1)([c_2])$. However  $H(\mathcal{F}_{C'}^1)([c_2])\in LCH_{\varepsilon_0}^{-1}(\Lambda_0)\simeq \{0\}$. Thus $[a_0]=H(\mathcal{F}_{C'}^1)([a_5])\circ 0=0$, this contradicts the existence of $\mathcal{F}_{C'}$ and hence the existence of $C'$. Thus $\Lambda\not\prec\Lambda_0$.

\end{proof}

\bibliographystyle{plain}
 \bibliography{Bibliographie_en}

\begin{thebibliography}{10}
\providecommand{\url}[1]{\texttt{#1}}
\providecommand{\urlprefix}{URL }
\providecommand{\eprint}[2][]{\url{#2}}
\providecommand{\MR}[2][]{MR #2}
\providecommand{\Zbl}[2][]{Zbl #2}

\bibitem{SivekBaldwin}
J.~Baldwin and S.~Sivek, Contact invariants in sutured monopole and instanton
  homology,{\it in preparation}. ----.

\bibitem{augcat}
F.~{Bourgeois} and B.~{Chantraine}, {Bilinearised Legendrian contact homology
  and the augmentation category}. \textit{ArXiv e-prints, to appear in "Journal
  of Symplectic Geometry"}  (2012).
\eprint{1210.7367}

\bibitem{chantraine_conc}
B.~Chantraine, Lagrangian concordance of {L}egendrian knots. \textit{Algebr.
  Geom. Topol.} \textbf{10} (2010), 63--85.
 \MR{2580429 (2011f:57049)}

\bibitem{collarable}
B.~Chantraine, {Some non-collarable slices of Lagrangian surfaces.}
  \textit{Bull. Lond. Math. Soc.} \textbf{44} (2012), 981--987.

\bibitem{Chekanov_DGA_Legendrian}
Y.~Chekanov, Differential algebra of {L}egendrian links. \textit{Invent. Math.}
  \textbf{150} (2002), 441--483.
 \MR{MR1946550 (2003m:53153)}

\bibitem{atlas}
W.~Chongchitmate and L.~Ng, An atlas of {L}egendrian knots. \textit{Exp. Math.}
  \textbf{22} (2013), 26--37.
 \MR{3038780}

\bibitem{Ekholm_Rational_SFT}
T.~Ekholm, Rational symplectic field theory over {$\Bbb Z_2$} for exact
  {L}agrangian cobordisms. \textit{J. Eur. Math. Soc. (JEMS)} \textbf{10}
  (2008), 641--704.
 \MR{2421157 (2009g:53130)}

\bibitem{Ekholm_Contact_Homology}
T.~Ekholm, J.~Etnyre, and M.~Sullivan, The contact homology of {L}egendrian
  submanifolds in {${\mathbb{R}}\sp {2n+1}$}. \textit{J. Differential Geom.}
  \textbf{71} (2005), 177--305.
 \MR{MR2197142}

\bibitem{Ekhoka}
T.~{Ekholm}, K.~{Honda}, and T.~{K{\'a}lm{\'a}n}, {Legendrian knots and exact
  Lagrangian cobordisms}. \textit{ArXiv e-prints}  (2012).
\eprint{1212.1519}

\bibitem{findimlagint}
Y.~Eliashberg and M.~Gromov, Lagrangian intersection theory: finite-dimensional
  approach. In \textit{Geometry of differential equations}, Amer. Math. Soc.
  Transl. Ser. 2 186, Amer. Math. Soc., Providence, RI 1998, 27--118.
 \MR{1732407 (2002a:53102)}

\bibitem{Eliashmurphy}
Y.~Eliashberg and E.~Murphy, Lagrangian caps. \textit{Geometric and Functional
  Analysis}  (2013), 1--32.

\bibitem{Eliashberg_&_Local_Lagrangian_knots}
Y.~Eliashberg and L.~Polterovich, Local {L}agrangian {$2$}-knots are trivial.
  \textit{Ann. of Math. (2)} \textbf{144} (1996), 61--76.
 \MR{MR1405943 (97g:58055)}

\bibitem{Geiges_Intro}
H.~Geiges, \textit{An introduction to contact topology}. Cambridge Studies in
  Advanced Mathematics 109, Cambridge University Press, Cambridge 2008.
 \MR{2397738 (2008m:57064)}

\bibitem{Kalman_One_parameter_family}
T.~K{\'a}lm{\'a}n, Contact homology and one parameter families of {L}egendrian
  knots. \textit{Geom. Topol.} \textbf{9} (2005), 2013--2078 (electronic).
 \MR{MR2209366}

\bibitem{Murphy}
E.~{Murphy}, {Loose Legendrian Embeddings in High Dimensional Contact
  Manifolds}. \textit{ArXiv e-prints}  (2012).
\eprint{1201.2245}

\bibitem{Ngcomputable}
L.~Ng, Computable {L}egendrian invariants. \textit{Topology} \textbf{42}
  (2003), 55--82.
 \MR{1928645 (2003h:57038)}

\bibitem{Rolfsen_Knots}
D.~Rolfsen, \textit{Knots and links}. Mathematics Lecture Series 7, Publish or
  Perish Inc., Houston, TX 1990. Corrected reprint of the 1976 original.
 \MR{MR1277811 (95c:57018)}

\bibitem{Sivek}
S.~Sivek, Monopole {F}loer homology and {L}egendrian knots. \textit{Geom.
  Topol.} \textbf{16} (2012), 751--779.
 \MR{2928982}

\end{thebibliography}
\end{document}